\documentclass[a4paper, 10pt, leqno,final]{article}

\usepackage[T1]{fontenc}		     
\usepackage[utf8]{inputenc}			 
\usepackage[english]{babel}
\usepackage{libertine}

\usepackage{amsmath}
\usepackage{amssymb}
\usepackage{amsthm}
	\theoremstyle{plain}
		\newtheorem{mainthm}{\textsc{Theorem}}		

		\newtheorem{thm}{Theorem}[section]	
		\newtheorem{cor}[thm]{Corollary}	 
		\newtheorem{lem}[thm]{Lemma}		
		\newtheorem{prop}[thm]{Proposition}

	\theoremstyle{definition}
		\newtheorem{defn}[thm]{Definition}	

	\theoremstyle{remark}
		\newtheorem{rem}[thm]{Remark}		
		\newtheorem{note}[thm]{Notation}		

\numberwithin{equation}{section}	

\usepackage{mathtools}		
\usepackage{mathrsfs}		
\usepackage{eucal}			

\usepackage{braket}			
\usepackage[all,pdf]{xy}		
\usepackage{mparhack}		
\usepackage{relsize}			

\usepackage{a4wide}		

\usepackage{booktabs}		
\usepackage{multirow}		
\usepackage{caption}		
\usepackage{rotating}
\usepackage{subfig}

\usepackage{varioref}		
\usepackage{footmisc}

\usepackage{graphicx}		
\usepackage{epsfig}

\usepackage{enumerate}		
\usepackage[colorlinks=true, linkcolor=black, citecolor=black, 
urlcolor=blue]{hyperref}		
\mathtoolsset{showonlyrefs}
\newcommand{\GL}{\mathrm{GL}}
\newcommand{\SSS}{\mathbf{S}}		
\newcommand{\R}{\mathbf{R}}	
\newcommand{\C}{\mathbf{C}}

\newcommand{\igeo}{\iota_{\textup{geo}}}		

\newcommand{\OO}{\mathrm{O}}

\newcommand{\trasp}[1]{{#1}^\mathsf{T}}	
\newcommand{\Lagr}{\Lambda}
\newcommand{\LP}{\mathcal{LP}}
\newcommand{\iCLM}{\iota^{\textup{CLM}}}		
\newcommand{\Mat}{\mathrm{Mat}}
\newcommand{\mmod}{\quad({\mathrm{mod}\ 2)}}
\newcommand{\Gr}{\mathrm{Gr}}
\newcommand{\St}{\mathrm{V}}
\newcommand{\im}{\mathrm{Ran}\,}
\newcommand{\dom}{\mathcal W\,}
\newcommand{\domtau}{\mathcal W_\tau\,}
\newcommand{\hil}{\mathcal H\,}
\newcommand{\hiltau}{\mathcal H_\tau\,}
\newcommand{\norm}[1]{\left\| #1 \right\|}		
\newcommand{\dueind}[3]{\iota\left( {#1},{#2};{#3} \right)}	
\newcommand{\iCLMnew}[3]{\iota^{\textup{CLM}}\left( {#1},{#2};{#3} \right)}	
	
\newcommand{\parity}[2]{\sigma\left( {#1},{#2} \right)}		
\newcommand{\N}{\mathbf{N}}		
\newcommand{\Z}{\mathbf{Z}}		
\DeclareMathOperator{\diag}{diag}		
\DeclareMathOperator{\sgn}{sgn}		

\renewcommand{\leq}{\leqslant}
\renewcommand{\geq}{\geqslant}

\renewcommand{\=}{\coloneqq}			

\newcommand{\email}[1]{\href{mailto:#1}{\textsf{#1}}}
\newcommand{\Id}{I}

\title{Bifurcation of heteroclinic orbits \\ via an\\  index theory}

\author{Xijun Hu\thanks{The author is partially supported by NSFC( No.11425105) 
and NCET.},  
Alessandro Portaluri
\thanks{The   
author is partially supported by the project ERC Advanced Grant 2013 
No.~339958 ``Complex Patterns for Strongly Interacting Dynamical Systems --- 
COMPAT”, by Prin 2015 ``Variational methods, with applications to 
problems in mathematical physics and geometry” No.~2015KB9WPT$\_$001 and 
by Ricerca locale 2015  ``Semi-classical trace formulas and their application 
in physical chemistry” No.~Borr$\_$Rilo$\_$16$\_$01 .} }
\date{\today}

\date{\today}
\begin{document}
 \maketitle
 
\begin{abstract}
Heteroclinic orbits for one-parameter families of nonautonomous vectorfields
appear in a very natural way 
in many physical applications. Inspired by  a recent bifurcation result for 
homoclinic trajectories of nonautonomous vectorfield 
proved by author in  \cite{Pej08}, we  define a new  $\Z_2$-index  and we construct a 
index theory for heteroclinic  orbits of nonautonomous vectorfield.  We prove an index theorem, by showing that,
under some standard transversality assumptions, the $\Z_2$-index is 
equal to the  {\em parity\/}, a  homotopy invariant for paths of  Fredholm operators of index 0. 
As a direct consequence  of the  index theory developed in this paper, we get a new 
bifurcation result for heteroclinic orbits.  

\end{abstract}


\tableofcontents


\section*{Introduction}

Index theory reveals is central role in many problems of Classical Mechanics 
like in the study of 
linear and spectral stability of periodic solutions of differential systems, 
in the investigations on the existence and 
multiplicity of  elliptic closed characteristics on energy hypersurfaces, in 
bifurcation theory etc. One-parameter 
families of vectorfields appear very naturally by 
linearizing a 1D partial differential 
evolution equation  along a  special solution like, for instance, a  travelling 
wave or a  steady solitary wave solution, etc. In all of these cases, in fact, the parameter 
actually is represented by the spectral parameter. 

In spite of the fact that in the Hamiltonian world many index theorems are 
available in the literature (cf. \cite{HP17} and references therein),  no results at all are 
known  in the  non-Hamiltonian case.
The most striking difference between the Hamiltonian and non-Hamiltonian case is 
that in the former there is a 
natural homotopy invariant which encodes the topology of the solutions space; it is 
defined in terms of the  fundamental solution of the 
Hamiltonian system known in 
literature as {\em Maslov-type index\/} (cf., for instance,  \cite{CLM94, RS93, LZ00} 
references therein) whilst in the 
latter case essentially no homotopy invariant were detected so far.
Nevertheless, some decades ago a $\Z_2$-homotopy invariant for paths of Fredholm 
operators of index 0 termed {\em parity\/} was defined in the non-Hamiltonian realm. 
(We refer the interested reader to  \cite{FP91a,FP91b,PR98}). 
In this respect, we have to mention that recently, authors in \cite{LZ00}  defined an integer-valued  
homotopy invariant  for path of essentially  hyperbolic operators by generalizing to this class of 
Fredholm operators the classical notion of spectral flow very well-known in the self-adjoint case. 
However, for this class of operators, no finite dimensional counterpart  (like the Maslov-type index) as 
been discovered. 

Inspired by the definition of the Evans function, one 
of the main purpose of this paper, is  to construct a $\Z_2$-homotopy invariant, in 
terms  of the determinant of a path of matrices naturally associated to 
an ordered pair of paths of  linear subspaces parametrized by a bounded interval.
To this pair of paths we naturally associate some  bundles (by pulling back the tautological bundle 
on the Grassmannian) and out of these  trivial bundles 
we construct a new bundle on $\SSS^1$ through the classical clutching procedure.

Our first main result Theorem \ref{thm:1.1} claims that, under suitable assumptions 
on the one-parameter family of nonautonomous 
vectorfields, the parity of the path of operators arising by linearizing 
the system 
along a solution, coincides with the $\Z_2$-index constructed in terms 
of the invariant unstable and stable subspaces. 

It is well-known that the parity plays a central in order to detect the bifurcation 
from the trivial branch 
(cf. \cite{FP91a, FP91b, PR98, Pej08}). A similar role  
is played by the spectral flows in the Hamiltonian case. (For further details, we refer the 
interested reader to  \cite{FPR99, PPT04, PP05, MPP07,PW13, 
PW14b} and references therein). 
As direct consequence of the  $\Z_2$ Index Theory  we prove a sufficient 
condition for detecting the  bifurcation along a trivial branch of heteroclinic orbits, 
in terms of the $\Z_2$-index.  We conclude by observing that the 
bifurcation result  proved in Theorem \ref{thm:1.2} is completely different in the essence from the 
main result, recently proved  by author in \cite{Pej08} in which the bifurcation was related to a 
non-trivial twist of the asymptotic stable and unstable bundles at infinity.


\section{Description of the problem and main results}\label{sec:intro}

Let $g$ be the one-parameter nonautonomous vectorfield defined 
by 
\[
g:[0,1]\times\R\times\R^n\ni (\lambda, t, z)\longmapsto g\big(\lambda, t, z\big)\in 
\R^n
\]
and we assume that $g$ as well as $ D_zg$ are bounded. Let $z_-, z_+ \in \R^n$ be two zeroes 
of $g$, meaning that 
$g(\lambda,t,z_-)=g(\lambda,t,z_-) =0$ for every $(\lambda, t) \in [0,1] \times \R$ 
and let us consider the 
first order differential system 
\begin{equation}\label{eq:1.1}
 \begin{cases}
  z'(t)- g(\lambda,t,z(t))=0, \qquad t \in \R\\
  \lim_{t \to -\infty}z(t)=z_- ,\quad  \lim_{t \to +\infty}z(t)=z_+. 
  \end{cases}
\end{equation}
We assume that  $z_\lambda$ is a solution of the system given in 
Equation \eqref{eq:1.1}. By linearising the vectorfield along $z_\lambda$, we get 
the following linear 
one-parameter family of first order systems 
\begin{equation}\label{eq:1.2}
 \begin{cases}
  u'(t)-S_\lambda(t)u(t)=0, \qquad t \in \R\\
  \lim_{t \to \pm\infty}u(t)=0 
  \end{cases}
\end{equation}
where we set $S_\lambda(t)\=  D_zg(\lambda,t,z_\lambda(t))$.
We introduce the following assumptions.
\begin{enumerate}
 \item[(A1)] The smooth family of matrices $S:[0,1] \times \R\to \Mat(n, \R)$ such 
that 
 $S_\lambda\=S(\lambda, \cdot): \R \to \Mat(n, \R)$ converges uniformly w.r.t.
$\lambda$ to families 
 \[
  S_\lambda^+= \lim_{t \to \infty} S_\lambda(t), \qquad S_\lambda^-=\lim_{t \to 
-\infty} S_\lambda(t), \quad 
  \lambda \in [0,1]. 
 \]
We assume that both $S_\lambda^+$ and $S_\lambda^-$ are hyperbolic, i.e.  the 
spectrum does not lie to the imaginary axis; namely
\[
\mathfrak{sp}(S^+_\lambda) \cap i\R= \mathfrak{sp}(S^-_\lambda) \cap i\R=\emptyset.
\]\footnote{
We recall that $T 
\in \Mat(n,\R)$ is termed 
{\em hyperbolic\/} if it has no eigenvalues on the imaginary axis. Thus in this 
case 
the spectrum  of a hyperbolic operator $T$ consists of two isolated closed 
components (one of 
which may be empty) 
\[
 \sigma(T) \cap \Set{z \in \C| \Re(z)<0} \textrm{ and } \sigma(T) \cap \Set{z 
\in \C| \Re(z)>0}.
\]}
\item[(A2)]  For $\lambda=0,1$, the system given in Equation \eqref{eq:1.2}  
only admits the 
trivial solution $u\equiv0$.
\item[(A3)]  For $\lambda\in [0,1]$, $\dim V^+(S_\lambda^-)=k$ and 
$\dim V^-(S_\lambda^+)=n-k$,  where $V^-$ and $V^+$ denote respectively the negative and positive 
spectral space with associated spectral projections 
\end{enumerate}
\begin{rem}\label{rem:importante}
It is worth noticing that assumption (A1) implies that  the families $\lambda 
\mapsto S^+_\lambda$ 
and $\lambda \mapsto S_\lambda^-$ are continuous.
\end{rem}
Let $\gamma_{(\lambda,\tau)}: \R \to \Mat(n,\R)$ be the  linear flow of the 
system given in Equation \eqref{eq:1.2}, i.e. the  (fundamental) matrix-valued  
solution of the linear initial value problem 
\begin{equation}\label{eq:fund-sol}
 \begin{cases}
\gamma'_{\lambda,\tau}(t)=S_\lambda(t)\, \gamma_{(\lambda,\tau)}(t), \qquad t 
\in 
\R\\
\gamma_{(\lambda,\tau)}(\tau)=\Id.
 \end{cases}
\end{equation}
We recall that the {\em stable\/} and {\em unstable\/} subspaces of the linear 
system given in 
Equation \eqref{eq:1.2} are  
\begin{equation}
E_\lambda^s(\tau)\=\Set{v \in \R^{2n}|\lim_{t\to 
+\infty}\gamma_{(\lambda,\tau)}(t)\, v=0} \textrm{ and }
E_\lambda^u(\tau)\=\Set{v \in \R^{2n}|\lim_{t\to 
-\infty}\gamma_{(\lambda,\tau)}(t)\, v=0}.
\end{equation}
By invoking \cite[Proposition 1.2]{AM03}, Assumption (A1), implies the 
following  uniform convergence result on the invariant manifolds
\[
 \lim_{\tau \to +\infty}E^s_\lambda(\tau)=V^-(S^+_\lambda) 
\textrm{ and }  
 \lim_{\tau \to -\infty} E^u_\lambda(\tau)=V^+(S^-_\lambda),
 \]
where the convergence is meant in the (gap-metric) topology of the  Grassmannian 
manifold; furthermore 
for any  fixed $\tau_0$,  $\lambda\mapsto E^s_\lambda(\tau_0)$ and $\lambda 
\mapsto  E^u_\lambda(\tau_0)$ 
are continuous. 
\begin{note}
We denote by $\pitchfork$ the transversality between two  linear subspaces 
meaning that their sum is the whole space. In symbols, if $V, W \subset \R^{n}$, 
$V \pitchfork W $ means that $ V+W=\R^n.$
\end{note}
Under the non-degeneracy   assumption  given in (A2), we get that 
\[
E^s_1(0)\pitchfork E^u_1(0) \textrm{ and  }
E^s_0(0)\pitchfork E^u_0(0).
\]
In particular the  $\Z_2$-index  $\dueind{E^s_1(0)}{ E^u_1(0)}{\lambda\in [0,1]}$ (cf. 
Definition \ref{def:z2index}), is well-defined.

We set $\dom\=W^{1,2}(\R,\R^n)$ and $\hil \=L^2(\R,\R^n)$. Under the  assumptions 
(A1) and (A3) it is well-known  (cf. 
\cite[Proposition 3.1]{Pej08}) that for 
each $\lambda \in [0,1]$, the operator 
\[
A_\lambda=\frac{d}{dt}-S_\lambda: \dom\longmapsto \hil
\]
is Fredholm of  index $0$. Moreover, by assumption (A2), the operators  $A_0, 
A_1$ are both 
invertibles. Thus it remains well-defined a homotopy invariant known in 
literature as   {\em parity\/}, namely 
$\parity{A_\lambda}{\lambda \in [0,1]} \in \widetilde{KO}([0,1], \partial([0,1]))\cong 
\Z_2$, where  
$ \widetilde{KO}$ denotes the reduced Grothendieck group. We refer  to \cite{FP91a,FP91b} 
and references therein, for further details. The first main result of the 
present paper is a sort of (mod 2) {\em   spectral flow formula\/}. 
\begin{mainthm} \label{thm:1.1}{\bf( A $\Z_2$-index theorem)\/} 
Under the assumptions (A1)-(A2) and (A3), the following equality holds
\begin{equation}\label{eq:1.3}
\parity{A_\lambda}{\lambda \in [0,1]}=\dueind{E^s_1(0)}{E^u_1(0)}{\lambda\in [0,1]}.
\end{equation}
  \end{mainthm}
A direct application of Theorem \ref{thm:1.1} is in bifurcation theory since its non-triviality  is 
sufficient in order to 
detect bifurcation.  It is well-known, in fact,  that (cf. \cite[Theorem 6.1]{PR98}), that 
the nontrivial parity implies bifurcation from the trivial branch.
More precisely, a point $\lambda_* \in [0,1]$ 
is a {\em bifurcation point\/} for heteroclinic solutions of the system given in 
Equation \eqref{eq:1.1} from the trivial branch $\lambda \mapsto z_\lambda$, if there 
exists a sequence $(\lambda_k, z_k)_{k \in \N} \subset [0,1] \times \dom $, where 
$z_k$ are solutions of Equation \eqref{eq:1.1}, $z_k \neq z_\lambda$ and 
$(\lambda_k, z_\lambda) \to (\lambda_*, z_{\lambda_*})$.   
  \begin{mainthm}\label{thm:1.2}
Let $z_\lambda$ be a family of (heteroclinic) solutions of the system given in 
Equation \eqref{eq:1.1} where 
the restpoints $z_-$ and $z_+$ are hyperbolic. We assume that 
$\dueind{E^s_\lambda(0)}{ E^u_\lambda(0)}{\lambda \in [0,1]}\neq 0$. Then for all 
$\epsilon>0$ sufficiently small there exists a (nontrivial) 
solution $(\lambda,z)$ of the system given in Equation \eqref{eq:1.1} 
such that $\|z-z_\lambda\|_{W^{1,2}}=\epsilon$.
\end{mainthm}
For $\lambda =0,1$ the systems are termed  {\em boundary non degenerate\/}, 
if the following transversality condition holds
\[
V^-(S^+_0)\pitchfork V^+(S^-_0)  \textrm{ and } V^-(S^+_1)\pitchfork V^+(S^-_1). 
\footnote{It is worth noticing that in the case of homoclinic orbit this condition is equivalent 
to the hyperbolicity of the equilibrium point.}
\]
In this case it is possible to associate to the heteroclinic orbit  $z_0$ (resp. $z_1$) 
the  $\Z_2$-index, termed 
{\em geometrical parity\/}, $\igeo(z_0)$ 
that counts mod 2 the number of nontrivial intersections between the path of 
stable and unstable subspaces parametrized by $\tau \in \R$. 
(We refer the reader to Definition \ref{def:index-hetero} and 
Definition \ref{def:index-heteroclinic-vero} for the rigorous statements). As 
direct consequence of  the homotopy invariance of this index, we immediately get  
\begin{equation}\label{eq:1.4}
\dueind{E^s_1(0)}{E^u_1(0)}{\lambda\in[0,1]}\equiv\igeo(z_0)+\igeo(z_1)+\dueind{
V^-(S^+_\lambda)}{ V^+(S^-_\lambda)}{\lambda \in [0,1]} \mmod
\end{equation}
In the special case in which both paths $\lambda \mapsto V^-(S^+_\lambda)$ and $\lambda \mapsto
V^+(S^-_\lambda)$ are independent on $\lambda$, then $\dueind{
V^-(S^+_\lambda)}{ V^+(S^-_\lambda)}{\lambda \in [0,1]}=0$. Thus in this case the bifurcation is 
detected by the following  condition 
\[
 \igeo(z_0)\not \equiv\igeo(z_1)\mmod.
\]


\section{A new  index for heteroclinic 
orbits and the geometrical parity}\label{sec:spectral-flow-and-Maslov}

The aim of this Section is to define the $\Z_2$-index, the {\em geometrical 
parity\/} of a 
non-degenerate heteroclinic orbit as well as to listen their basic properties. 

We start by briefly recalling some basic facts about the Grassmannian and to fix our notations. 
(For  all of this we refer the interested reader to the beautiful book 
\cite[Chapter 5]{MS78}).  We denote by $\Gr_k(n, \R)$ the set of all $k$-dimensional  linear 
subspaces of $\R^n$. As homogeneous space, it is well-known that 
\[
 \Gr_k(n,\R)=\OO(n)/\big(\OO(k) \times \OO(n-k)\big)
\]
where $\OO$ denotes the orthogonal group. In particular,  $\Gr_k(n,\R)$ is a 
$k(n-k)$-dimensional compact smooth manifold 
(in general it has the structure of a smooth algebraic variety) whose 
topology  is induced by the gap-metric 
\[
 d(V,W) \= \norm{P_V-P_W}
\]
where $P_V$ and $P_W$ denote the orthogonal projections in $\R^n$ 
onto the subspaces $V$ and $W$, respectively. 
A $k$-frame in $\R^n$ is a $k$-tuple of linearly independent vectors of $\R^n$ 
and the collections of all 
$k$-frames form an open subset of the $k$-fold Cartesian product of $\R^n$ 
called the {\em Stiefel 
manifold\/} and denoted by $\St_k(n,\R)$. There is a canonical function 
$q: \St_k(n,\R) \to \Gr_k(n,\R)$ which maps each $k$-frame to the $k$-dimensional 
linear subspace it spans. 
By \cite[Lemma 5.1, pag.31-33]{MS78} we also 
get that the correspondence $X \to X^\perp$ which assign to each 
$k$-dimensional linear subspace the 
$(n-k)$-dimensional linear orthogonal subspace, defines a homeomorphism between 
the $\Gr_k(n,\R)$ and $\Gr_{n-k}(n,\R)$. We denote by  $\gamma^k(n,\R)$  the 
{\em  tautological line  bundle\/} (or {\em universal bundle\/}) over the Grassmannian 
manifold $\Gr_k(n,\R)$.  Let  $V\in \mathscr C^0\big([a,b],\Gr_k(n,\R)\big)$ and  $W\in \mathscr 
C^0\big([a,b],\Gr_{n-k}(n,\R)\big)$,  and we assume the following 
transversality condition at the endpoints 
\begin{equation}\label{eq:transversality}
V(a)\pitchfork W(a) \textrm{ and }  V(b)\pitchfork W(b).
\end{equation}
For every $t \in [a,b]$, let 
$\mathcal E_V(t)\= \{ v_1(t), \dots v_k(t)\}$ and $\mathcal E_W(t)\= 
\{w_{1}(t),\cdots,w_{n-k}(t)\}$ 
be two frames generating $V(t)$ and $W(t)$ respectively. 
We consider  $M \in \mathscr 
C^0\big([a,b], \Mat(n, \R)\big)$ 
whose columns are given by $v_j$ and $w_l$; i.e. 
\[
M(t)\=\Big[v_1(t)\big|\cdots  \big|v_k(t) \big|w_1(t)\big|\cdots \big| w_{n-k}(t)\Big].
\]
By the transversality assumption given in Equation \eqref{eq:transversality}, 
it readily follows that 
the endpoints of the path $M$, namely $M(a), M(b)$ are nondegenerate matrices (in the 
sense that the determinant of 
$M(a)$ and $M(b)$ do not vanish). Thus we are entitled to introduce the 
definition of the $\Z_2$-index. 
\begin{defn}\label{def:z2index} 
We  term {\em $\Z_2$-index of the pair $V$ and $W$\/}, the integer 
\begin{equation}\label{eq:eq:index-z2}
 \dueind{V(t)}{W(t)}{t \in [a,b]}\=
 \begin{cases}
                                       0 \quad \textrm{ if $\det\big( M(a)\cdot 
M(b)\big) >0$ }\\
                                       1 \quad \textrm{ if $\det\big( M(a)\cdot 
M(b)\big)<0$ }
                                      \end{cases}. 
\end{equation}
\end{defn}
\begin{lem}
The $\Z_2$-index given in Definition \ref{def:z2index} is well-posed. 
\end{lem}
\begin{proof}
By a straightforward calculation, it readily follows that this  definition  is 
independent on the choice of the  frames. Let 
$\widehat{\mathcal E_V}$ and $\widehat{\mathcal E_W}$ 
be two continuous  frames for $V$ and $W$ respectively, pointwise 
given by 
$\widehat{\mathcal E_V}(t)\=\{\widehat v_1(t), \cdots \widehat v_k(t)\}$ and 
$\widehat{\mathcal E_W}(t)\=\{\widehat w_1(t), \cdots \widehat w_{n-k}(t)\}$  
and let us define the continuous path of  matrices 
\[
\widehat M(t)\=
\Big[\widehat v_1(t)\big|
\cdots \big| \widehat v_k(t)\big|\widehat w_1(t)\big| \cdots \big|
\widehat w_{n-k}(t)\Big].
\]
Thus, for every $t \in [a,b]$,  there exists $G_1(t) \in \GL(k)$ and $G_2(t) 
\in \GL(n-k)$ such that 
$\widehat M(t)=M(t)G(t)$ for $G(t)\= \diag\big(G_1(t), G_2(t)\big)$. In 
particular, 
$G(t)$ is  nondegenerate for every $t \in [a,b]$ and  
$\sgn\big(\det(G(t)\big)$ is independent 
on $t$. Thus  
\begin{multline}
\sgn\Big(\det\big( \widehat{M}(a)\cdot \widehat{M}(b)\big)\Big)=\sgn\Big(\det\big( 
M(a)\cdot M(b)\big)\Big)\sgn\Big( \det\big( G(a)\cdot G(b)\big)\Big)\\
= \sgn\Big(\det\big( 
M(a)\cdot M(b)\big)\Big).
\end{multline}
This conclude the proof.
\end{proof}
We now list some properties (omitting the proofs) of the $\Z_2$-index which are 
straightforward consequences of Definition \ref{def:z2index}.
\subsection*{\bf Properties of the $\Z_2$-index\/}

\begin{enumerate}
 \item[]{\bf Property I. (Reparametrisation Invariance)\/} 
 Let $\psi:[c,d] \to [a,b]$ be a  continuous function such that  $\psi(c)=a$ 
and $\psi(d)=b $, or $\psi(c)=b$ and $\psi(d)=a $ . 
 Then 
 \[ 
 \dueind{V(t)}{W(t)}{t \in [a,b]}= \dueind{(V\circ\psi)(t)}{(W\circ\psi )(t)}{t 
\in [a,b]}.
 \]
 \item[]{\bf Property II. (Homotopy invariance Relative to the Ends)\/}   Let 
\[
[0,1]\times [a,b] \ni (s,t)\mapsto (V_s(t), W_s(t))\in \Gr_k(n,\R)\oplus 
\Gr_{n-k}(n,\R)
\]
be a continuous two-parameter family subspaces such that $V_s(a)\pitchfork 
W_s(a)$ and $V_s(b)\pitchfork W_s(b)$.  Then 
\[ \dueind{V_0(t)}{W_0(t)}{t \in [a,b]}=\dueind{V_1(t)}{W_1(t)}{t \in [a,b]}.
\]
\item[]{\bf Property III. (Path Additivity)\/} If $c \in (a,b)$ such that 
$V(c)\pitchfork W(c)$ , then
\[
\dueind{V(t)}{W(t)}{ t\in[a,b]}\equiv \dueind{V(t)}{W(t)}{ t\in[a,c]}+
\dueind{V(t)}{W(t)}{t\in[c,b]} \quad \mmod
\]
\item[]{\bf Property IV. (Symmetry property)\/}   
\[
\dueind{V(t)}{W(t)}{t \in [a,b]}=\dueind{W(t)}{V(t)}{t \in [a,b]}.
 \]
 \item[]{\bf Property V. (Sum Additivity)\/} For $i=1,2$, let  $(V_i,W_i)\in 
\Gr_{k_i}(n_i, \R)\oplus \Gr_{n_i-k_i}(n_i,\R)$. Then 
 \begin{multline}
 \dueind{(V_1\oplus V_2)(t)}{(W_1\oplus W_2)(t)}{t \in 
[a,b]}\equiv\dueind{V_1(t)}{W_1(t)}{t \in [a,b]}\\
 + 
 \dueind{V_2(t)}{W_2(t)}{t \in [a,b]} \quad \mmod
 \end{multline}
 \end{enumerate}
\begin{rem}
We remark that the $\Z_2$-index defined above actually depends on 
the 
whole path and not just on its endpoints. 
\end{rem}
\begin{note}
 We denote by $\mathcal P([a,b],k,n)$ (resp. by $ \mathcal P^*([a,b],k,n)$) 
 the  space  of all ordered pairs  of 
 continuous paths of subspaces (resp. with transversal ends) 
 \[
  \mathcal P([a,b],k,n)=\{Z \in \mathscr C^0\big([a,b], \Gr_k(n,\R) \times 
\Gr_{n-k}(n,\R)\big)|
  Z(t)=\big(V(t), W(t)\big)\}
 \]
and we let 
\[
   \mathcal P^*([a,b],k,n)=\{Z \in \mathcal P([a,b],k,n)| V(a)\pitchfork W(a) 
\textrm{ and } 
   V(b)\pitchfork W(b)\}.
\]
\end{note}

\subsection{A Vector Bundle over the circle}

The aim of this subsection is to construct a vector bundle over $\SSS^1$ whose 
triviality is determined by the vanishing of the 
$\Z_2$-index. 
Given the ordered pair of subspace $(V,W) \in 
\mathcal{P}^*([0,1], k,n)$, we  define the 
path $\widetilde W: [0, 2]\to \Gr_k(n,\R)$ as follows
\[
 \widetilde W(t)= 
\begin{cases}
W(t) & \textrm{ for } t \in [0,1]\\
W(2-t) & \textrm{ for } t \in [1, 2].
\end{cases}
\]
\begin{rem}
 Actually the path $\widetilde W$ on the interval $[1,2]$  geometrically coincides with $W$ 
 (on the interval $[0,1]$) but travelled in the opposite direction.  
\end{rem}
\begin{lem}\label{thm:costruzione-V-tilde}
There exists a 
continuous path $\widetilde  V: [0,2] \to \Gr_k(n,\R) $ such that 
\begin{enumerate}
 \item $\widetilde V|_{[0,1]}=V$.
 \item $\widetilde V$ is closed, namely   $\widetilde V(0)=\widetilde V(2)$.
 \item For every $t\in [0,1]$ the following transversality condition holds 
 \[
  \widetilde V(t) \pitchfork \widetilde W(t) \qquad t \in [1, 2].
 \]
\end{enumerate}
\end{lem}
\begin{proof}
We start to define the 
path $\widetilde V:[0,2] \to \Gr_k(n,\R)$ 
\begin{equation}\label{eq:de-tilde-V}
 \widetilde V(t)= 
\begin{cases}
V(t) & \textrm{ for } t \in [0,1]\\
\widehat V(t) & \textrm{ for } t \in [1, 2]
\end{cases}
\end{equation}
where $\widehat V$ is such that $\widehat V(1)= V(1)$ and $\widehat V(2)= 
V(0)$.  Clearly the path $\widetilde  V$ 
given in Equation \eqref{eq:de-tilde-V} is closed, being $\widetilde V(0)=V(0)$ 
and $\widetilde V(2)= \widehat 
V(2)= V(0) $. From the definition, it holds also that 
$\widetilde V|_{[0,1]}=V$; furthermore 
it is easy to check that  $\widetilde V(1) \pitchfork \widetilde W(1)$ and 
$\widetilde V(2) \pitchfork \widetilde W(2)$. These last two facts readily 
follows from the definitions of 
$\widetilde V$ and $\widetilde W$ and from the fact that the ends of the two 
paths $V$ and $W$ are transversal. Now,  
if $\widetilde V(t) \pitchfork \widetilde W(t)$ for all $t \in [1,2]$ the 
result follows. 
If not, it just enough to observe that the path $\widetilde W^\perp:[1, 2] 
\to \Gr_k(n,\R)$ pointwise 
given by the orthogonal complement to $W$ clearly satisfies the following 
transversality condition
 \[
   \widetilde W^\perp (t) \pitchfork \widetilde W(t) \qquad t \in [1,2].
 \]
Thus, if $\widetilde V(1)= \widetilde W^\perp(1)$ and $\widetilde V(2)= 
\widetilde W^\perp(2)$, it is 
just enough to define  $\widehat V(t)= \widetilde W^\perp(2-t)$ for all $t\in 
[1,2]$. Otherwise, we reduce to the 
previous situation as follows. For $\epsilon \in (0,1)$, let us consider the 
continuous path
\begin{equation}
\widehat V(t)\=\begin{cases}
\widetilde W^\perp(1+\epsilon)  \textrm{ for all } t\in 
[1,1+\epsilon]\\
\widetilde W^\perp(2-t) \textrm{ for all } t\in 
[1+\epsilon,2-\epsilon]\\
\widetilde W^\perp(2-\epsilon)  \textrm{ for all } t\in 
[2-\epsilon,2 ].
\end{cases} 
\end{equation}
By choosing $\epsilon>0$ sufficiently small and observing that the transversality is an open condition, 
the result readily follows. This conclude the proof. 
\end{proof}
By passing to the quotient  of $[0,2]$ with respect to its boundary, the 
function $\widetilde V:[0,2] \to \Gr_k(n,\R)$ induces a map, that with a slight abuse of notation, we 
still will denote by the same symbol, 
$\widetilde V: \SSS \to \Gr_k(n,\R)$ for $\SSS\=\R/(2\, \Z)$. 
Let $\pi: E\big(\gamma^k(n,\R)\big)\to 
\Gr_k(n,\R)$ be  denote the (standard) tautological bundle projection onto its first factor. 
In shorthand notation we set $E(\gamma^k_n)\=E\big(\gamma^k(n,\R)\big)$. 
We now consider the pull-back bundle  of $E(\gamma^k_n)$ through 
$\widetilde V$; thus we have the following  commutative diagram
\[
\xymatrix{
{\widetilde V}^*\big(E(\gamma^k_n)\big) 
\ar[r]^-{\mu} \ar[d]_-{{\widetilde V}^*(\pi)} & E(\gamma^k_n) \ar[d]_{\pi}    \\
\SSS\ar[r]_-{\widetilde V} & \Gr_k(n,\R)
}
\]
where as usually, $\widetilde V^*(\pi)$ denotes the pull-back projection $\pi$ through $\widetilde V$.
The next  result gives a necessary and sufficient condition on the triviality 
of the  pull-back of the tautological bundle induced by $\widetilde V$ in terms of the 
triviality of 
the $\Z_2$-index. 
\begin{lem}\label{lem2.1}
The vector bundle constructed by pulling back the tautological bundle through $\widetilde V$ 
is trivial if only if 
\[
\dueind{V}{W}{t\in [0,1]}=0.
\]
\end{lem}
\begin{proof} We start to observe that as direct consequence of third property 
stated in Lemma 
\ref{thm:costruzione-V-tilde} as well as by the homotopy invariance of the 
$\Z_2$-index with 
respect to its ends, we get that 
\begin{equation}\label{eq:triv-1}
\dueind{V}{W}{ t\in [0,1]}=\dueind{\widetilde V}{\widetilde W}{ t\in [0,2]}.
\end{equation}
We define the constant path $\widehat W:[0,2] \to \Gr_k(n,\R)$ 
as follows $\widehat{W}(t)\equiv W(0) $ and again as consequence of  the homotopy 
invariance property of the $\Z_2$-index, we have
\begin{equation}\label{eq:triv-2}
\dueind{\widetilde V}{\widetilde W}{ t\in [0,2]}=\dueind{\widetilde 
V}{\widehat W}{ t\in [0,2]}.
\end{equation}
We consider the $k$-frame $\mathcal F_V(t)=\{e_{1}(t)),\cdots,e_{k}(t)\}$ for the 
subspace $V(t)$ and the (constant)  $(n-k)$-frame 
$\mathcal F_{W(0)}=\{ e_{k+1},\cdots,e_{n}\}$  for $W(0)$.
As before, we  define the $n \times n$  matrix 
\[
M(t)\=\Big[e_{1}(t)\big|\cdots\big|e_{k}(t)\big| e_{k+1}\big|\cdots\big|e_{n} \Big].
\]
It is immediate to observe that  the vector bundle over $\SSS$, namely  
${\widetilde V}^*(\pi):{\widetilde V}^*\big(E(\gamma^k_n)\big) \to \SSS$ is 
trivial  if and only if $\det\big(M(0)\cdot M(2)\big)>0$, which is 
equivalent to state 
that $\dueind{\widetilde V}{\widehat W}{ t\in [0,2]}=0$. Now, the conclusion 
readily follows by invoking  
Equations \eqref{eq:triv-1}-\eqref{eq:triv-2}. 
\end{proof}


\subsection{A new  index for heteroclinic orbits of nonautonomous vectorfields}

This subsection is to define a $\Z_2$-index  in the case of 
heteroclinic orbits of  a one-parameter family of nonautonomous systems. 
We start by setting   $n=2k$ and to consider the  symplectic real vector  
space $(\R^{2k},\omega)$ where $\omega$ is the standard symplectic form.
We denote by $\Lagr(k)$ the Lagrangian Grassmannian manifold, namely the set 
of all Lagrangian subspaces of $(\R^{2k},\omega)$.  It is well-known that it is  a 
real 
compact and connected analytic $\frac12 k(k+1)$-dimensional submanifold of the 
Grassmannian manifold $\Gr_k(2k,\R)$. 
\begin{note}
 We denote by $\LP([a,b],k)$ the space  of all ordered pairs 
of continuous paths of Lagrangian subspaces 
 \[
  \LP([a,b],k)\=\Big\{Z \in \mathscr C^0\big([a,b], \Lagr(k) \times 
\Lagr(k) \big)| Z(t)=\big(V(t), W(t)\big)\big)\Big\}
 \]
and we let 
\[
   \LP^*([a,b],k)\=\Big\{Z \in \LP([a,b],k)| V(a)\pitchfork W(a) 
\textrm{ and } V(b)\pitchfork W(b)\Big\}.
\]
\end{note}
To each pair ordered pair of paths of Lagrangian subspaces  
\[
(V,W):[a,b] \ni t \mapsto \big(V(t), W(t)\big) \in  \Lagr(k),
\]
we associate 
a homotopy invariant known in literature as {\em Maslov index\/}, 
that will be denoted by 
\[
\iCLMnew{V(t)}{W(t)}{t \in [a,b]}.
\]
(We refer the 
interested  reader to   the beautiful papers \cite{CLM94,RS93}.)  
\begin{prop}\label{prop:2.1} 
Let $(V,W) \in  \LP^*([a,b],k)$. Then we have
\[ 
\dueind{V(t)}{W(t)}{t \in [a,b]}\equiv\iCLMnew{V(t)}{W(t)}{t \in [a,b]} \quad \mmod
. 
\]
\end{prop}
\begin{proof}
Following authors in \cite[Section 4]{CLM94}, up to a slight perturbation, 
we can assume that the crossing instants (intersections) between the paths $V$ 
and $W$ are simply, meaning that 
they are 1-dimensional and transversal. By codimensional arguments, this is 
generically true, and by 
the invariance property of  the index $\iCLM$ (with free endpoints in the case 
of transversal ends), 
is actually independent on this choice.  We assume that 
$t_0 \in (a,b)$ is a crossing 
instant such that $\dim\big(V(t_0)\cap W(t_0)\big)=1$. There exists 
$\varepsilon 
>0$ such that for 
$t \in [t_0 -\varepsilon, t_0+ \varepsilon)]$, $V(t)= \Gr\big(A(t)\big)$ and 
$W(t)= \Gr\big(B(t)\big)$ where $A$ 
and $B$ are smooth path of symmetric matrices (generating the Lagrangian 
subspaces). In this case, following authors in \cite[Theorem 3.1]{LZ00} and 
\cite[Section 3]{RS93}, we get that 
\[
\iCLMnew{V(t)}{W(t)}{t \in [a,b]}= \sum_{t_0 \in (a,b) } \sgn \Gamma( W ,V ,t_0)
\]
where $\Gamma$ denotes the (non-degenerate) crossing (quadratic) form on 
$W(t)\cap V(t)$,  
$\sgn$ denotes  the sign and the sum runs all over the crossing instants which, 
by the non-degeneracy 
assumption   are isolated; thus  on a compact 
interval are in a finite number. (Cf. \cite[Definition 3.2]{RS93}, for further details).
In order to conclude the proof, it is 
enough to prove that the 
local contribution to the $\iCLM$ as well as to $\iota$ coincide. By a direct 
computation 
it follows that the crossing form at the crossing instant $t_0$ is given by 
\begin{equation}
\Gamma\big(W,V, t_0\big) : \ker\big(B(t_0) - A(t_0)\big) \longrightarrow\R: 
u \longmapsto \Gamma\big(W,V, t_0\big)[u]= \langle \big[\dot B(t_0)-\dot A(t_0)\big] u, u 
\rangle. 
\end{equation}
By invoking Kato's selection theorem, 
$B(t)\cong \diag\big(\lambda_1(t), \dots, \lambda_k(t)\big)$ where, for $j 
=1, \dots k$,  
$\lambda_j(t)$ represent the repeated eigenvalues of $B(t)-A(t)$ according to 
its own multiplicity.  

Since $\dim\ker\big(B(t)-A(t)\big)=1$, there exists only one  
changing-sign eigenvalue at $t_0$; let's say $\lambda_i$.   Thus, we get that 
\begin{equation}
\sgn\Gamma\big(W,V, t_0\big)
= \begin{cases}
1 &  \textrm{ if  } \dot \lambda_i(t_0)>0\\
-1 & \textrm{ if  } \dot \lambda_i(t_0)<0
\end{cases}      
\end{equation}
To conclude the proof we  now define the matrix 
\begin{equation}
 M(t)\= \begin{bmatrix}
         \Id & \Id \\
         A(t)& B(t)
        \end{bmatrix}
\end{equation}
having nullity (i.e. dimension of the kernel) precisely 1 and let $C$ be the block upper 
triangular matrix defined by 
$
C\= \begin{bmatrix}
         \Id & -\Id \\
         0 & \Id
        \end{bmatrix}
$. We observe 
\begin{equation}
M(t) \cdot C\= \begin{bmatrix}
         \Id & 0\\
         A(t) & B(t)-A(t)
        \end{bmatrix}; \quad  \textrm{ thus  we get } 
\end{equation}
$\det \big(M(t)C\big)= \det\big( M(t)\big)= \det 
\big(B(t)-A(t)\big)$. In particular,
 $\ker M(t_0)= \ker\big(B(t_0)-A(t_0)\big)$ (which is $1$-dimensional). By 
this arguments and by taking 
into account Definition \ref{def:z2index}, we get that 
\begin{equation}
\dueind{V(t)}{W(t)}{t \in [t_0-\varepsilon, t_0+ \varepsilon]}= 1
\end{equation}
This conclude the proof. 
\end{proof}
From now on, we assume that $[a,b]$ is an unbounded interval (either $a=-\infty$ 
or $ b=+\infty$). Thus we have the following three kind of unbounded intervals,  
namely 
$(-\infty, b]$, $[a,+\infty)$ and finally $(-\infty,+\infty)$. 
We assume that there exists $T>0$ such that $V(t) \pitchfork W(t)$ for every 
$t \leq -T$,  $t \geq T$ and finally $|t| \geq T$, in the first, second and finally in the  third case  respectively. 
\begin{defn}\label{def:z2indexTT}
Under the previous notation, we define  the $\Z_2$-index   as 
follows: 
\begin{multline}\label{eq:index-tagliato}
 \dueind{V(t)}{W(t)}{t \in (-\infty, b] }\= \dueind{V(t)}{W(t)}{t \in [-T,b]}\\ 
 \dueind{V(t)}{W(t)}{t \in [a,+\infty)}\= \dueind{V(t)}{W(t)}{t \in [a,T]}\\
  \dueind{V(t)}{W(t)}{t \in (\infty, +\infty) }\= \dueind{V(t)}{W(t)}{t \in [-T,T]}.
\end{multline}
\end{defn}
\begin{rem}
Directly by the definition and by the path additivity property of the $\Z_2$-index, 
it readily follows that the  it is well-defined in the sense that  it is 
independent on $T$.
\end{rem}
Let $S: \R \to \Mat(n,\R)$ be  a a continuous path of matrices and we assume 
that there exist 
$S^{\pm}$ which are hyperbolic and such that 
\[
\lim_{t\to\pm\infty}S(t)=S^\pm
\]
For $\tau \in \R$, we let $\gamma_\tau: \R \to \Mat(n,\R)$ be the associated 
matrix-valued solution such that 
$\gamma_\tau(\tau)=\Id$, and  we 
denote by  $E^s(\tau)$ and $ E^u(\tau)$ respectively the stable and unstable 
subspace. 
By invoking \cite[Proposition 2.1]{AM03}, we immediately get the following 
convergence result 
\begin{equation}\label{eq:convergence-stable-unstable}
 \lim_{\tau \to +\infty} E^s(\tau)= V^-(S^+) \textrm { and } 
\lim_{\tau \to 
-\infty} E^u(\tau)= V^+(S^-).
\end{equation}
\begin{defn}\label{def:index-hetero} 
Under the previous notation, we assume the following transversality condition is fulfilled 
\[
V^+(S^-) \pitchfork V^-(S^+) \textrm{ and } E^s(0)\pitchfork E^u(0).
\]
We define the {\em $\Z_2$-index of the  path $S$\/}, as follows
\[ 
\iota(S)\= \dueind{E^s(t)}{E^u(-t)}{ t\in [0,+\infty)}.
\]
\end{defn}
\begin{rem}
 We observe that by taking into account the convergence stated in Equation 
 \eqref{eq:convergence-stable-unstable} as well as Definition 
\ref{def:z2indexTT}, the index 
 given in Definition \ref{def:index-hetero} is well-defined. 
\end{rem}
Thus we are entitle to introduce the following definition. 
\begin{defn}\label{def:index-heteroclinic-vero}
Let $z_ -$ and  $z_+$ two hyperbolic restpoints. We term {\em  geometrical parity of 
the heteroclinic orbit $x$\/} connecting them,  the $\Z_2$-index 
of the linear  path  $S$ arising by linearizing  the nonautonomous vectorfield along $x$; thus 
in symbol
  \[
   \igeo(x)\= \iota(S)
  \]
  where $\iota(S)$ is given in Definition \ref{def:index-hetero}.
 \end{defn}
In the special case in which the system is Hamiltonian, 
as direct consequence of Proposition \ref{prop:2.1} as well as Definition 
\ref{def:z2indexTT}, 
Definition \ref{def:index-hetero} and finally Definition 
\ref{def:index-heteroclinic-vero}, we get 
the following result. 
\begin{cor}\label{def:z2indexT}  
Let $x$ be heteroclinic solution of the nonautonomous Hamiltonian  vectorfield between 
the hyperbolic 
restpoints $z_-$ and $z_+$. Thus, we have
\[
\iota(x)=\iCLM(x)  \mmod .
\]
\end{cor}
By the homotopy invariance of the $\Z_2$-index, we get the 
following result . (Cf. \cite{HP17}, for further details). 
\begin{prop}\label{thm:prop2.1} 
Let us consider the system  given in Equation \eqref{eq:1.2} and we assume 
conditions (A1)-(A2)-(A3). If 
\[
V^-(S_0^+)\pitchfork V^+(S_0^-)  \textrm { and } 
V^-(S_1^+)\pitchfork V^+(S_1^-) 
\]
then,  we have 
\[ 
\dueind{E_\lambda^u(0)}{E_\lambda^s(0)}{ \lambda\in[0,1]}\equiv 
\igeo(S_0)+\igeo(S_1)+\dueind{V^-(S_\lambda^+}){V^+(S_\lambda^-)}{ 
\lambda\in[0,1]}  \mmod .    
\]
\end{prop}


\section{Parity for path of Fredholm operator and the Index theorem }

In this section  we introduce  
the other last main ingredient of the  $\Z_2$-index prove Theorem \ref{thm:1.1} and Theorem \ref{thm:1.2}.

Let $X,Y$ be two real and separable Hilbert spaces and  $\mathcal{F}_0(X,Y)$ 
be denote 
the set of all Fredholm  operators of index $0$. Given a continuous path  
$T: [0,1]\to \mathcal{F}_0(X,Y)$ having  invertible endpoints,  
there is  a homotopy invariant of $T$ termed {\em parity of $T$\/}  and denoted 
by 
$\parity{T(\lambda)}{\lambda \in [0,1]}$ which is an 
element of the (reduced real Groethendieck group), $\widetilde{KO}$; in symbols
\[
\parity{T(\lambda)}{\lambda \in [0,1]}\in \widetilde{KO}([0,1], \partial([0,1]))\cong \Z_2.
\]
In a geometrical fashion, the parity can be generically seen as a mod 2 
intersection index between a continuous path in $\mathcal F_0(X,Y)$ and the one-codimensional 
submanifold of all degenerate operators. 
(For further details, we refer the interested reader to 
\cite[Section 3]{FP91b}).   As  was proved by 
authors in  \cite[Section 2]{FP91a}, if $T:P \to \mathcal F_0(X,Y)$ is a continuous family of 
linear Fredholm operators of index $0$ parametrised by $P$, the parity is equivalent  to the 
nonorientability of the  index bundle (actually an equivalence class of vector bundles) 
of the path $T$, namely $\mathrm{Ind}(T) \in \widetilde{KO}(P)$ and 
in particular it is measured by  $w_1(\mathrm{Ind}(T)) \in H^1(P,  \Z_2)$ 
the {\em first Stiefel-Whitney  class\/} of $\mathrm{Ind}(T)$.

For the sake of the reader, we explain what the parity of $T$ means,  in the special 
case of finite-dimensional vector space  and for continuous families parametrised by $\SSS^1$.  
We recall that  vector bundles over the spheres, could be constructed by means of the trivial bundle 
on disks (homeomorphic to the upper lower hemisphere), through the clutching functions. More 
precisely,  let  $E_2, E_1$ be two real vector bundles over $[0,1]$ such that  $\dim E_2=\dim 
E_1=k$ and  let $T: E_2\to  E_1$ be a bundle morphism such that  $T|_{\partial([0,1])}$ is 
invertible. We set 
$I_j\=[0,1] \times\{j\}$ for $j=0,1$ and we consider the disjoint union  
\[
I_0\coprod I_1
\]
obtained by  identifying the four points  (two by two) in $\partial([0,1]) \times \{0\}$ and 
$\partial([0,1]) \times \{1\}$ as follows 
\[
(0,0)\cong (0,1) \textrm{ and }(1,0)\cong (1,1).
\]
Under this identification,  $\partial([0,1]) \times \{0\} \cup \partial([0,1]) \times 
\{1\}$ is 
homeomorphic to $\SSS^1$. 
Since by assumption $T|_{\partial([0,1])}$ is a bundle isomorphism, we get (for any such a bundle map $T$),  
a  well-defined bundle $E_T$ over $\SSS^1$. We have 
\[
\parity{T(\lambda)}{\lambda \in [0,1]}\=w_1(E_T)=\begin{cases}
                                       0 \quad \textrm{ if $E_T$ is orientable }\\
                                       1 \quad \textrm{ if $E_T$ is unorintable }.
                                      \end{cases} 
\]
We recall that the two generators of $ \widetilde{KO}([0,1], \partial([0,1]))\cong  
\widetilde{KO}(\SSS^1)\cong \Z_2$ are the trivial and the M\"obius bundle. 
Since every vector bundle on the interval is trivial, up to identifying the 
fibres with the Euclidean 
space $\R^k$, the  bundle map $T$ induces a continuous 
path of linear maps on $\R^k$ parametrised by $[0,1]$ and pointwise 
defined by $T(s): \R^k\to \R^k$,  such that  $T(0), T(1)$ are invertible. 
\begin{prop}\label{thm:cor2.3}
For any $s \in [0,1]$, let   $\Gr\big(T(s)\big)\subset \R^k \times \R^k$ be denote  
the graph of $T_s$  and let $\Lambda_0=\R^k\times \{0\}$.  Then, we  have
\begin{equation}
\parity{T(s)}{s \in [0,1]}=\dueind{\Gr\big(T(s)\big)}{\Lambda_0}{s \in [0,1]}. 
\end{equation}
\end{prop}
\begin{proof}
Let $\mathscr E\=\{e_1,\cdots,e_{2k}\}$  be the canonical basis of $\R^k \times 
\R^k$; thus in particular  $\Lambda_0$ is generated by $\{e_1, \dots, 
e_k\}$. We define the 
following block matrix 
\begin{equation}
M(s)\=
\begin{bmatrix}
\Id & \Id\\
 T(s) &0 
\end{bmatrix}, \qquad s \in [0,1]
\end{equation}
and we observe also that $\det M(s)= (-1)^k \det T(s)$. In particular, we have
\begin{equation}
\dueind{\Gr\big( T(s)\big)}{\Lambda_0}{s \in [0,1]}=
\begin{cases}
1 &  \textrm{ if } \det T(0)\cdot \det T(1) >0\\
0 & \textrm{ if } \det T(0)\cdot \det T(1) <0.
\end{cases}
\end{equation}
Furthermore, it is obvious that  the constructed bundle $E_T$ is orientable 
if and only if  $\det T_0\cdot\det T_1>0$. This complete the proof. 
\end{proof}
We are now in position to discuss the infinite dimensional case. Since $[0,1]$ is 
compact, there exists a subspace $V\subset Y$, such that 
\begin{equation}\label{eq:eq2.1FP} 
\im T(s) +V=Y,\quad  s \in [0,1].  
\end{equation}
We let $E_s\=T^{-1}(s)(V)$ and we observe that   $\dim E_s=\dim V$. From 
Equation \eqref{eq:eq2.1FP} 
it easily follows that the set $E\=\{(s, x)| T(s)x \in V\}$ is the total 
space of a (trivial) vector bundle over $[0,1]$. 
\begin{defn}\label{def:parity}
Let $X,Y$ be two real and separable Hilbert spaces and let  
\[
T: \big([0,1], \partial([0,1])\big) \to \big(\mathcal{F}_0(X,Y), \GL(X,Y)\big)
\]
be a continuous path of pairs.
We define the {\em parity\/}  of $T$, $
\parity{T(s)}{s \in [0,1]}\in \widetilde{KO}([0,1], \partial([0,1]))\cong \Z_2$ 
as the first Stiefel-Whitney class of the bundle $E_T \to \SSS^1$
\[
\parity{T(s)}{s \in [0,1]}=w_1\big(E_T\big).
\]
\end{defn}
\begin{rem}
It is well-known that this definition is well-posed in the sense that it is 
independent on the 
choice of $V$. (Cf. \cite{FP91a, FP91b} and references therein). 
\end{rem}
Let us consider the (smoooth) path of first order differential operators, pointwise 
defined by 
\[
A_\lambda\=\dfrac{d}{dt}-S_\lambda(t), \qquad t \in \R
\]
arising by the system given in Equation \eqref{eq:1.2}.
and we observe (cf. \cite[Proposition 3.1]{Pej08}) 
that for each $\lambda \in [0,1]$, the operator $A_\lambda$ is a bounded Fredholm 
operator of index 0 from 
$\dom$ into $\hil$.  We recall that under the assumption (A2), both the operators 
$A_0$ and $ A_1$ are  invertible. For $\tau>0$, we  let 
\[ 
E_\lambda(\tau)\=\Set{x\in W^{1,2}([-\tau,\tau], \mathbb{R}^n)|x(-\tau)\in 
E^u(-\tau) ,  x(\tau) \in E^s(\tau)}
\] 
and we define   $A_{\lambda,\tau}$ to be  the restriction of   
$A_\lambda$  to $E_\lambda(\tau)$, namely 
\[
 A_{\lambda,\tau}\= A_\lambda\big\vert_{E_\lambda(\tau)}.
\]
We recall that the adjoint $A_\lambda^*$ of $A_\lambda$  
is the (closed) unbounded operator on $\hil$ densely defined 
on $\dom$ given  by 
\begin{equation}\label{eq:adjoint-do}
 A_\lambda^*\=-\dfrac{d}{dt}-\trasp{S}_\lambda(t).
\end{equation}
As before, for $\tau>0$, we define the subspace 
\[ 
F_\lambda(\tau)\=\Set{x\in W^{1,2}([-\tau,\tau], \mathbb{R}^n)|x(-\tau)\in 
\big[E^u(-\tau)\big]^\perp ,  x(\tau) \in \big[E^s(\tau)\big]^\perp }
\] 
and we denote   by $A^*_{\lambda,\tau}$ to be  the restriction of   
$A^*_\lambda$  to $F_\lambda(\tau)$, namely 
\[
 A^*_{\lambda,\tau}\= A^*_\lambda\big\vert_{F_\lambda(\tau)}.
\]
Given  $x\in  E_\lambda(\tau)$ (resp. $x\in  F_\lambda(\tau)$)  we extend $x$ on the whole of $\R$ as 
follows; we let 
 \begin{equation}\label{eq:estensione}
e_{\lambda,\tau}(x)(t)\=\begin{cases}
                                       \gamma_{(\lambda,\tau)}(t)x(\tau) &
\textrm{ if } t\geq \tau \\
                                       x(t) & \textrm{ if }
t\in[-\tau,\tau]   \\
                                         \gamma_{(\lambda,-\tau)}(t)x(-\tau) 
& \textrm{ if } t\leq -\tau .
                                      \end{cases}
\end{equation}
where $\gamma_{(\lambda, *)}$ is the matrix-valued solution of the system  
defined in Equation \eqref{eq:1.2} (resp. in Equation \eqref{eq:adjoint-do}).   It 
such that  $\gamma_{(\lambda, *)}(*)=\Id$. It is immediate to check that 
$x\in \ker A_{\lambda,\tau}$ (resp. $x\in \ker A^*_{\lambda,\tau}$ )
if and only if  $e_{\lambda,\tau}(x)\in \ker 
A_\lambda$  (resp. $x\in \ker A^*_{\lambda}$) and by this claim the next result readily follows. 
  \begin{lem}\label{thm:equivgsp}
The following equality  holds
\begin{equation}
\begin{split}
\dim\ker A_\lambda=\dim\ker A_{\lambda,\tau}=
\dim\ker A^*_\lambda &=\dim\ker A^*_{\lambda,\tau}.
\end{split}
\end{equation}
\end{lem}
The path $\lambda \mapsto A_{\lambda, \tau}$ plays a central role in the next Proposition which 
represent the main ingredient for proving Theorem \ref{thm:1.1}.
\begin{prop}\label{thm:3.3}  For $\tau>0$ sufficiently large, the following 
equality holds: 
\[
\parity{A_{\lambda,\tau}}{\lambda \in [0,1]}=\parity{A_{\lambda}}{\lambda \in [0,1]}. 
\]
\end{prop}
\begin{rem}
The main idea behind the proof of Proposition \ref{thm:3.3} is that,
if $\tau>0$ is sufficiently large, the 
arising vector bundles constructed through $A_\lambda$ and $A_{\lambda, \tau}$ are (bundle) 
isomorphic; thus the first Stiefel-Whitney classes  coincide. 
\end{rem}
\begin{note}
In what follows, we set
\begin{equation}
\begin{split}
\domtau \=W^{1,2}([-\tau, \tau], \R^n)&\\
& \hiltau= L^2([-\tau, \tau], \R^n)
\end{split}
\end{equation}
\end{note}
We define the   following  {\em restriction\/} and {\em prolongation operator\/} respectively 
denoted by $\chi$ and $E_2$ and defined as  follows 
 \[
 \chi_\tau: \hil \longrightarrow\hiltau: x(t)\longmapsto 
x(t)|_{[-\tau,\tau]}
\]  
and $ E_2: \hiltau\to \hil$ defined by setting 
  \[
E_2(x)(t)\=
\begin{cases}
0 & \textrm{ if }|t|\geq \tau\\
x(t) & \textrm{ if } t\in[-\tau,\tau].
\end{cases}
\]
\begin{lem}\label{thm:primo}
For each $\lambda \in [0,1]$, there exists a finite dimensional subspace 
$V_\lambda$ of  $\dom$, such that 
\begin{equation}
 V_\lambda + \im A_\lambda = \hil.
\end{equation}
\end{lem}
\begin{proof}
We let  $\lambda_0 \in (0,1)$.
For $\epsilon >0$ sufficiently small,  we consider the interval $[\lambda_0-\epsilon, 
\lambda_0+\epsilon]\subset [0,1]$. 
If $\epsilon>0$ is small enough, then there exists a finite dimensional 
subspace $V_\lambda$ of 
$ \hil$
such  that $\ker A^*_{\lambda}\subset V_\lambda$.
To see this, we let $\widetilde A_\lambda\=A_\lambda A_\lambda^*$ and we observe 
that $\lambda \mapsto \widetilde 
A_\lambda$ is a continuous path of selfadjoint Fredholm operators  such that 
\[
\ker \widetilde A_\lambda=  \ker A_\lambda^*.
\]
Let $\Gamma$ be a small circle around the origin chosen in such a 
way  $\Gamma\cap\sigma(\widetilde A^*_\lambda)=\emptyset $ for each  $\lambda\in[\lambda_0-\epsilon, 
\lambda_0+\epsilon]$ 
and let us consider  the projector operator 
\[
P_\lambda=\frac{1}{2\pi i}\int_\Gamma\big[z-(\widetilde A^\C)^*_{\lambda}\big]^{-1}dz , 
\]  
and  $V_\lambda=\im P_\lambda\supset \ker \widetilde A_\lambda=  \ker A_\lambda^*$. 
Since $\lambda \mapsto P_\lambda$ is 
continuous on  $[\lambda_0-\epsilon, 
\lambda_0+\epsilon]$, it follows that  the path $\lambda \mapsto V_\lambda$ is. 
Moreover by the Fredholmness of $A_\lambda^*$,  we get that $\dim V_\lambda <+\infty$.
\end{proof}
\begin{lem}\label{thm:secondo}
We let $V_{\lambda, \tau}\=\chi_\tau V_\lambda$ where $V_\lambda \in \hil$ 
has been defined in   Lemma \ref{thm:primo}.  Thus, we have
\begin{equation}
  V_{\lambda, \tau} + \im A_{\lambda,\tau} = \hiltau.
\end{equation}
\end{lem}
\begin{proof}
The proof of this result follows directly by arguing as in the proof of Lemma \ref{thm:primo} once 
observed that 
\[
 V_{\lambda, \tau} \supset \ker A^*_{\lambda, \tau}.
\]
In fact, if $u \in \ker A_{\lambda, \tau}^*$, then 
$e_{\lambda, \tau}(u) \in \ker A_\lambda^* \subset V_\lambda$. In order to conclude, it is enough to 
observe that $u= \chi_\tau(e_{\lambda, \tau} u) \in V_{\lambda, \tau}$. 
\end{proof}
\begin{lem}\label{thm:terzo}
There exists $\tau_0>0$ such that, if $\tau \geq \tau_0$, then we have 
\begin{equation}
 \norm{\chi_\tau(y)}_2 \geq \dfrac12 \norm{y}_2, \qquad \forall  y \in V_\lambda.
\end{equation}
Thus the linear map $\chi_\tau: V_\lambda \to V_{\lambda, \tau}$ is injective (and hence an isomorphism). 
\end{lem}
\begin{proof}
For each $\lambda \in [0,1]$, let $P_\lambda: \hil \to \hil$ be the projector operator onto 
(the finite-dimensional vector space) $V_\lambda$. Let $\{e_1, \dots, e_n\}$ be a unitary 
$n$-frame for $\im P_0$ and, for each $i \in \{1,\dots, n\}$, we let 
\[
 e_i(\lambda)\=P_\lambda e_i.
\]
Thus, there exists $\epsilon_1 >0$ sufficiently small such that 
$\{e_1(\lambda), \dots,e_n(\lambda)\}$ is a $n$-dimensional frame which depends continuously on 
$\lambda \in [0,\epsilon_1]$. We denote by $V$ the total space of the (trivial) vector bundle over $[0,1]$ 
and  by $S_{[0,\epsilon_1]}(V)$ the total space of the sphere bundle over $[0,\epsilon_1]$. We now 
consider the continuous function 
\[
 f_\tau: S_{[0,\epsilon_1]}(V) \to \R \textrm{ defined by } f_\tau(y)\=\norm{\chi_\tau(y)}_2.
\]
By compactness of  $S_{[0,\epsilon_1]}(V)$, the function $f_\tau$ is actually uniformly continuous and 
by the very  definition of $\chi_\tau$, we infer that  
\[
 \lim_{\tau \to +\infty} f_\tau(y)=1. 
\]
By  uniformly continuity of $f_\tau$, we get that there exists $\tau_1>0$ sufficiently large such that 
\[
 f_\tau(y) \geq \dfrac12 \qquad \textrm{ for every } y \in S_{[0,\epsilon_1]}(V) \textrm{ and for every } 
 \tau \geq \tau_1 .
\]
By compactness of $[0,1]$ and by choosing an $\epsilon_0>0$ (maybe smaller than $\epsilon_1$), 
there exists $\tau_0>0$ such that 
\[
 f_\tau(y) \geq \dfrac12 \qquad \textrm{ for every } y \in 
 S_{[0,1]}(V)\textrm{ and } 
 \tau \geq \tau_0
\]
and by this the thesis readily follows. This conclude the proof. 
\end{proof}
\begin{lem}\label{thm:quarto}
We let   $K_{\lambda, \tau}\= E_2\big( \chi_\tau \ker A_\lambda^*\big)$ and let $\tau_0$ as in 
Lemma \ref{thm:terzo}. Then, for every  $\tau \geq \tau_0$,  we get 
 \begin{itemize}
  \item[(i)] $\dim \ker A_\lambda^*= \dim K_{\lambda, \tau}$
  \item[(ii)] $K_{\lambda, \tau} \cap \im A_\lambda =\{0\}$.
 \end{itemize}
\end{lem}
\begin{proof}
In order to prove the first item, it is enough to observe that by taking into account 
Lemma \ref{thm:terzo},  the map 
$\chi_\tau:V_\lambda \to V_{\lambda, \tau}$ is injective and hence also its 
restriction on $\ker A_\lambda^*$. 

In order to prove the second statement, we argue by contradiction as follows. If not, 
there exists $0 \neq y=A_\lambda x$ such that $y \in K_{\lambda, \tau}$. In particular, by definition 
of $K_{\lambda, \tau}$, there exists $\bar y \in \ker A_\lambda^*$ such that $y=E_2\chi_\tau(\bar y)$. 
We now observe that 
\begin{equation}\label{eq:vv1}
 \langle y, \bar y\rangle_2= \langle A_\lambda x, \bar y\rangle_2 = 
 \langle  x, A_\lambda^* \bar y\rangle_2=0
\end{equation}
where the last equality follows by the fact that $\bar y \in \ker A_\lambda^*$. Furthermore 
\begin{equation}\label{eq:vv2}
 \langle y, \bar y\rangle_2= \langle E_2\chi_\tau(\bar y), \bar y \rangle_2
 =\norm{\chi_\tau(\bar y)}^2 \neq 0. 
\end{equation}
Summing up Equation \eqref{eq:vv1} and Equation \eqref{eq:vv2} we get a contradiction. This conclude the 
proof. 
\end{proof}

\begin{lem}\label{thm:quinto}
For $\tau>0 $ sufficiently large there exists $\overline V_\lambda \subset \hil$ finite-dimensional 
subspace such that 
 \begin{equation}
  \overline V_\lambda + \im A_\lambda =\hil
 \end{equation}
\end{lem}
\begin{proof}
By 
taking into account Lemma \ref{thm:equivgsp} we infer also that $\chi_\tau(\ker A_\lambda^*)= 
\ker A_{\lambda, \tau}^*\subset V_{\lambda, \tau}$ and by invoking Lemma \ref{thm:quarto} we infer that for 
every $\tau \geq \tau_0$ we get 
\[
 E_2(V_{\lambda, \tau}) + \im A_\lambda = \hil.
\]
In order to conclude the proof, it is just enough to define 
$\overline V_\lambda\=  E_2(V_{\lambda, \tau_0})$. 
\end{proof}
\paragraph*{Proof of Proposition \ref{thm:3.3}}
By te previous Lemmas, we deduce  that there exists $V_{\lambda,\tau}$ 
and $\overline V_\lambda$ finite dimensional subspaces respectively of $\hiltau$ and $\hil$ such that 
\[
 V_{\lambda, \tau} + \im A_{\lambda, \tau}= \hiltau \textrm{ and } \overline V_\lambda + 
 \im A_\lambda = \hil. 
\]
We now set $E_\lambda=A^{-1}_{\lambda,\tau}V_{\lambda,\tau}$  and 
$\overline E_\lambda=A^{-1}_{\lambda}\overline V_\lambda$. Let $y \in V_{\lambda, \tau}$ and 
$A_{\lambda, \tau}(x)=y$; in particular  $x \in \domtau$ is such that 
\[
\begin{cases}
 \dfrac{d x}{d\tau} -S_\lambda(t)x=y,\quad t \in [-\tau,\tau]\\
 x(-\tau) \in E^u_\lambda(-\tau), \quad  x(\tau) \in E^s_\lambda(\tau).
\end{cases}
\]
We now set $\bar y\= E_2 y$ and we let $A_\lambda^{-1}(\bar x)= \bar y$. By this, readily follows that 
$\bar x = e_{\lambda, \tau} x$ where $e_{\lambda, \tau}$ has been defined in Equation \eqref{eq:estensione}.
By this argument it follows that $e_{\lambda, \tau}: E_{\lambda, \tau} \to \overline  E_\lambda$ is an 
isomorphism with inverse $\chi_\tau: \overline E_\lambda \to E_{\lambda, \tau}$.  We 
also observe that $\chi_\tau: \overline V_\lambda \to V_{\lambda, \tau}$ is an isomorphism with inverse 
$E_2: V_{\lambda, \tau} \to \overline V_\lambda$. 
By the  commutativity of the following diagram 
\[
\xymatrix{
\overline E_\lambda \ar[r]^-{A_\lambda} \ar@/^/[d]^-{\chi_\tau} & \overline V_\lambda 
\ar@/^/[d]^-{\chi_\tau}    \\
E_{\lambda, \tau} \ar@/^/[u]^-{e_{\lambda, \tau}}\ar[r]_-{A_{\lambda,\tau}} & \overline V_{\lambda, \tau}
 \ar@/^/[u]^-{E_2}
}
\]
we get that $w_1(\overline E_\lambda)= w_1(E_{\lambda, \tau})$ and this conclude the proof. \qed

Let us consider the following change of variables obtained by setting 
$s=\dfrac{t+\tau}{2\tau}\in[0,1]$. Then,   the operator 
$A_{\lambda,\tau}$ can be rewritten as follows
\[
A_{\lambda,\tau}=\dfrac{1}{2\tau}\dfrac{d}{ds}-S_\lambda\big((2s-1)\tau\big): 
\mathcal W_{\lambda}(\tau)\subset L^2([0,1];\R^{n})\to L^2([0,1];\R^{n}), 
\]
where 
\[
\mathcal W_{\lambda}(\tau):=\{u\in W^{1,2}([0,1],\R^n) | u(0)\in 
E^u_\lambda(-\tau) \textrm{ and } u(1)\in E^s_\lambda(\tau)
\}. 
\]
For each $\lambda \in [0,1]$, we define the following operator 
\[
\bar{A}_{\lambda,\tau}=\dfrac{d}{ds}-2\tau S_\lambda\big((2s-1)\tau\big)
\]
and we observe that, in contrast with $A_{\lambda, \tau}$,  
$\overline A_{\lambda,\tau}$ is well-defined also for $\tau=0$; in fact 
$\overline A_{\lambda,0}=\dfrac{d}{dt}$ on the  domain $\mathcal W_{\lambda}(0)$ and by taking 
into 
account the homotopy invariant of the parity, we get  
\[
\parity{\bar{A}_{\lambda,\tau}}{\lambda \in 
[0,1]}=\parity{A_{\lambda,\tau}}{\lambda \in [0,1]}.  
\]
\begin{lem}\label{lem:3.4}  
The following equality holds: 
\[ 
\parity{\bar{A}_{\lambda,0}}{\lambda \in 
[0,1]}=\dueind{E^u_\lambda(0)}{E^s_\lambda(0)}{\lambda \in [0,1]}. 
\]
\end{lem}
 \begin{proof}   We start to observe that  since 
 $\overline A^*_{\lambda}(0)=-\dfrac{d}{dt}$, its 
kernel  consists of all constant functions. Let us denote by $V$ be the space of all 
constant 
functions on $L^2([0,1],\R^n)$ and we observe that it is isomorphic to $\R^n$.  
By  a direct computation, we get that  the space 
$E_\lambda=\bar{A}_{\lambda,0}^{-1}V$ can be characterised as follows 
\[ 
 E_\lambda=\{z \in W^{1,2}([0,1], \R^n)| z(t)=x+t(y-x),  x\in E^u_\lambda(0), y\in 
E^s_\lambda(0)    \}, 
\] 
and $\bar{A}_{\lambda,0}z=y-x$.   Clearly we have the following isomorphism 
$E_\lambda\approx 
E^u_\lambda(0)\oplus E^s_\lambda(0)$, and 
 \[
 \bar{A}_{\lambda,0}: E^u_\lambda(0)\oplus E^s_\lambda(0)\to \R^n, (x,y)\to 
y-x.   \]
Let $\{e_1(\lambda),\cdots,e_k(\lambda)\}$ be a  frame  of $ 
E^u_\lambda(0)$ and 
$\{e_{k+1}(\lambda),\cdots,e_n(\lambda)\}$ be a  frame of $ 
E^s_\lambda(0)$. 
We set  
\[
T(\lambda)\=\Big[e_1(\lambda)\big|\cdots\big|e_k(\lambda)\big|
e_{k+1}(\lambda)\big|\cdots\big|e_n(\lambda)\Big].
\]
Then,  we have 
\[
\dueind{E^u_\lambda(0)}{E^s_\lambda(0)}{\lambda \in [0,1]}=0 \iff \det\big( 
T(0)\cdot T(1)\big)>0.
\]
We observe that every bundle on the interval is trivial and hence also 
the vector bundle $E$ over $[0,1]$ having fibre $E_\lambda$ is. In particular, the  
trivialisation map is given by 
\[
 [0,1] \times \R^n \ni (\lambda, e_1, \dots, e_n) \mapsto \big(e_1(\lambda), \dots, 
e_n(\lambda)\big).
\]
Thus  we have the isomorphism of $E^u_\lambda(0)\oplus 
E^s_\lambda(0)\equiv\R^n$ and,  under this trivialisation, we get  
\[
\overline A_{\lambda,0}=\Big[-e_1(\lambda)\big|\cdots\big|-e_k(\lambda)\big|
e_{k+1}(\lambda)\big|\cdots\big|e_n(\lambda)\Big]
\]  
and hence 
\[\sgn\big( \det \bar{A}_{0,0} \cdot \bar{A}_{1,0}\big)=
\sgn\big( \det T(0)\cdot T(1)\big)  \]
 which conclude the proof. 
 \end{proof} 
{\bf Proof of Theorem \ref{thm:1.1}}. 
It  readily follows by summing up the conclusion 
proved in Proposition 
\ref{thm:3.3} and Lemma \ref{lem:3.4}. This conclude the proof. \qed

 {\bf Proof of Theorem \ref{thm:1.2}} It follows by  taking into account 
 Theorem \ref{thm:1.1} and by invoking Theorem \ref{thm:1.1} and 
 \cite[Theorem 6.1]{PR98}. This conclude the proof. \qed


\vspace{1cm}
	\noindent
	\textsc{Prof. Xijun Hu}\\
	Department of Mathematics,
	Shandong University\\
	Jinan, Shandong, 250100 \\
	The People's Republic of China,
	China\\
	E-mail: \email{xjhu@sdu.edu.cn}

\vspace{1cm}
\noindent
\textsc{Prof. Alessandro Portaluri}\\
DISAFA,
Università degli Studi di Torino\\
Largo Paolo Braccini 2 \\
10095 Grugliasco, Torino, 
Italy\\
Website: \url{aportaluri.wordpress.com}\\
E-mail: \email{alessandro.portaluri@unito.it}

\vspace{1cm}
\noindent
COMPAT-ERC Website: \url{https://compaterc.wordpress.com/}\\
COMPAT-ERC Webmaster \& Webdesigner: Arch.  Annalisa Piccolo


\begin{thebibliography}{99}



\bibitem[AM03]{AM03}
{\sc Abbondandolo, Alberto; Majer, Pietro.}
\newblock  Ordinary differential operators in Hilbert spaces and Fredholm 
pairs. 
\newblock Math. Z. 243 (2003), no. 3, 525--562. 

\bibitem[BHPT17]{BHPT17}
{\sc Barutello Vivina; Hu Xijun;  Portaluri Alessandro; Terracini Susanna.}
\newblock An index theorem for colliding and parabolic motions in Celestial 
Mechanics.
\newblock Preprint


\bibitem[CLM94]{CLM94}
{\sc Cappell, Sylvain E.; Lee, Ronnie; Miller, Edward Y..}
\newblock On the Maslov index. 
\newblock Comm. Pure Appl. Math. 47 (1994), no. 2, 121--186.

\bibitem[CH07]{CH07}
{\sc   Chen, Chao-Nien; Hu, Xijun.}
\newblock Maslov index for homoclinic orbits of 
Hamiltonian systems. 
\newblock Ann. Inst. H. Poincaré Anal. Non Linéaire 24 (2007), no. 4, 589--603. 






\bibitem[FP91a]{FP91a}
{\sc Fitzpatrick, Patrick M.; Pejsachowicz, Jacobo.}
\newblock Nonorientability of the index bundle and several-parameter 
bifurcation. 
\newblock  J. Funct. Anal. 98(1) 1991, 42--58.


\bibitem[FP91b]{FP91b}
{\sc  Fitzpatrick, Patrick M.; Pejsachowicz, Jacobo.}
\newblock Parity and generalized multiplicity. 
\newblock Trans. Amer. Math. Soc. 326 (1991), No.1, 281--305.




\bibitem[FPR99]{FPR99}
{\sc Fitzpatrick, Patrick M.; Pejsachowicz, Jacobo; Recht, Lazaro.}
\newblock Spectral flow and bifurcation of critical points of 
strongly-indefinite functionals. I. 
General theory. 
\newblock J. Funct. Anal. 162 (1999), no. 1, 52--95. 


\bibitem[GGK90]{GGK90}
{\sc Gohberg, Israel; Goldberg, Seymour; Kaashoek, Marinus A..}
\newblock Classes of linear operators. Vol. I. 
\newblock Operator Theory: Advances and Applications, 49. 
Birkhäuser Verlag, Basel, 1990. xiv+468 pp.



\bibitem[HP17]{HP17}
{\sc Hu, Xijun; Portaluri Alessandro.}
\newblock Index theory for heteroclinic orbits of Hamiltonian Systems 
\newblock Preprint available on \url{http://arxiv.org/abs/1703.03908}

\bibitem[LZ00]{LZ00}
{\sc  Long, Yiming;  Zhu, Chaofeng.}
\newblock Maslov-type index theory for symplectic paths and spectral flow (II) 
\newblock Chinese Ann. Math. Ser. B 21 (2000), no. 1, 89--108.


\bibitem[MS78]{MS78}
{\sc Milnor, John W.; Stasheff, James D..}
\newblock  Characteristic classes. 
\newblock Annals of Mathematics Studies, No. 76. Princeton University Press, 
Princeton, N. J.; 
University of Tokyo Press, Tokyo, 1974. 


\bibitem[MPP07]{MPP07}
{\sc Musso, Monica; Pejsachowicz, Jacobo; Portaluri, Alessandro.}
\newblock Morse index and bifurcation of p-geodesics on semi Riemannian 
manifolds.
\newblock  ESAIM Control Optim. Calc. Var. 13 (2007), no. 3, 598--621.



\bibitem[Pej08]{Pej08}
{\sc Pejsachowicz, Jacobo.}
\newblock Bifurcation of homoclinics. 
\newblock Proc. Amer. Math. Soc. 136 (2008), No. 1, 111--118. 

\bibitem[PR98]{PR98}
{\sc Pejsachowicz, Jacobo; Rabier, Patrick J. .}
\newblock Degree theory for $C^1$ Fredholm mappings of index 0.
\newblock J. Anal. Math. 76 (1998), 289--319. 

\bibitem[PP05]{PP05}
{\sc Piccione, Paolo; Portaluri, Alessandro.}
\newblock A bifurcation result for semi-Riemannian trajectories of the Lorentz 
force equation.
\newblock J. Differential Equations 210 (2005), no. 2, 233--262. 

\bibitem[PPT04]{PPT04}
{\sc Piccione, Paolo; Portaluri, Alessandro; Tausk, Daniel V..}
\newblock Spectral flow, Maslov index and bifurcation of
semi-Riemannian geodesics.
\newblock Ann. Global Anal. Geom. 25 (2004), no. 2, 121--149.

\bibitem[Por11]{Por11}
{\sc Portaluri, Alessandro.}
\newblock A K-theoretical invariant and bifurcation for a parameterized family 
of functionals. 
\newblock  J. Math. Anal. Appl. 377 (2011), no. 2, 762--770.

\bibitem[PW13]{PW13}
{\sc Portaluri, Alessandro; Waterstraat, Nils.}
\newblock On bifurcation for semilinear elliptic Dirichlet problems and the 
Morse-Smale index theorem. 
\newblock J. Math. Anal. Appl. 408 (2013), no. 2, 572--575. 

\bibitem[PW14a]{PW14a}
{\sc Portaluri, Alessandro; Waterstraat, Nils.}
\newblock Bifurcation results for critical points of families of functionals. 
\newblock Differential Integral Equations 27 (2014), no. 3-4, 369--386.

\bibitem[PW14b]{PW14b}
{\sc Portaluri, Alessandro; Waterstraat, Nils.}
\newblock On bifurcation for semilinear elliptic Dirichlet problems 
on geodesic balls. 
\newblock J. Math. Anal. Appl. 415 (2014), no. 1, 240--246. 

\bibitem[RS93]{RS93}
{\sc Robbin, Joel; Salamon, Dietmar.} 
\newblock The Maslov index for paths. 
\newblock Topology 32 (1993), no. 4, 827--844. 


\bibitem[ZL99]{ZL99}
{\sc   Zhu, Chaofeng; Long, Yiming.}
\newblock Maslov-type index theory for symplectic paths and spectral flow. I. 
\newblock Chinese Ann. Math. Ser. B 20 (1999), no. 4, 413--424.



\end{thebibliography}
\end{document}